\documentclass[12pt]{aptpub}

\usepackage{amsmath,amstext,url}
\usepackage{cases}
\usepackage{color}
\usepackage{appendix}

\renewcommand{\shorttitle}
{ \underline{{Ergodic property for Galton-Watson processes in which individuals have variable lifetimes}}}

\numberwithin{equation}{section} 

\makeatletter \@addtoreset{equation}{section}

\makeatletter \@addtoreset{lemma}{section}

\makeatletter \@addtoreset{theorem}{section}

\makeatletter \@addtoreset{proposition}{section}

\makeatletter \@addtoreset{corollary}{section}

\makeatletter \@addtoreset{remark}{section}

\makeatletter \@addtoreset{definition}{section}

\makeatletter \@addtoreset{example}{section}

\begin{document}

\thispagestyle{firstpg}

\vspace*{1.5pc} \noindent \normalsize\textbf{\Large { Ergodic property for Galton-Watson processes in which individuals have variable lifetimes}} \hfill

\vspace{12pt} \hspace*{0.75pc}{\small\textrm{\uppercase{Jiangrui Tan
}}}\hspace{-2pt}$^{*}$, {\small\textit{Beijing Normal University }}

\hspace*{0.75pc}{\small\textrm{\uppercase{Junping Li}}}
\hspace{-2pt}$^{**}$, {\small\textit{Central South University }}

\par
\footnote{\hspace*{-0.75pc}$^{*}\,$Postal address:
 School of Mathematical Sciences, Laboratory of Mathematics and Complex Systems,
Beijing Normal University, Beijing, 100875, China. E-mail address:
670816453@qq.com }

\par
\footnote{\hspace*{-0.75pc}$^{**}\,$Postal address:
 School of Mathematics and Statistics, Central
South University, Changsha, 410083, China. E-mail address:
jpli@mail.csu.edu.cn }


\par
\renewenvironment{abstract}{%
\vspace{8pt} \vspace{0.1pc} \hspace*{0.25pc}
\begin{minipage}{14cm}
\footnotesize
{\bf Abstract}\\[1ex]
\hspace*{0.5pc}} {\end{minipage}}
\begin{abstract}
   This paper is concerned with an extended Galton-Watson process so as to allow individuals to live and reproduce for more than one unit time. We assume that each individual can live $k$ seasons (time-units) with probability $h_k$, and produce $m$ offspring with probability $p_m$ during each season. These can be seen as Galton-Watson processes with countably infinitely many types in which particles of type $i$ may only have offspring of type $i+1$ and type $1$. Let $\emph{\textbf{M}}$ be its mean progeny matrix and $\gamma$ be the convergence radius of the power series $\sum_{k\geq 0}r^k(\emph{\textbf{M}}^k)_{ij}$. We first derive formula of calculating $\gamma$ and show that $\gamma$, in supercritical case, is actually the extinction probability of a Galton-Watson process. Next, we give clear criteria for $\emph{\textbf{M}}$ to be $\gamma$-transient, $\gamma$-positive and $\gamma$-null recurrent from which the ergodic property of the process is discussed. The criteria for $\gamma$ and $\gamma$-recurrence of $\emph{\textbf{M}}$ rely on the properties of lifetime distribution which are easier to be verified than current results. Finally, we show the asymptotic behavior of the total population size of each type of individuals under certain conditions which illustrates the evolution of Galton-Watson process in which individuals have variable lifetimes.
\end{abstract}

\vspace*{12pt} \hspace*{2.25pc}
\parbox[b]{26.75pc}{{
}}
{\footnotesize {\bf Keywords:}
Galton-Watson process; variable lifetimes; convergence radius.
\par
\normalsize

\renewcommand{\amsprimary}[1]{
     \vspace*{8pt}
     \hspace*{2.25pc}
     \parbox[b]{24.75pc}{\scriptsize
    AMS 2000 Subject Classification: Primary 60J27 Secondary 60J35
     {\uppercase{#1}}}\par\normalsize}
\renewcommand{\ams}[2]{
     \vspace*{8pt}
     \hspace*{2.25pc}
     \parbox[b]{24.75pc}{\scriptsize
     AMS 2000 SUBJECT CLASSIFICATION: PRIMARY
     {\uppercase{#1}}\\ \phantom{
     AMS 2000
     SUBJECT CLASSIFICATION:
     }
    SECONDARY
 {\uppercase{#2}}}\par\normalsize}

\ams{60J27}{60J35}

\par
\vspace{5mm}
 \setcounter{section}{1}
 \setcounter{equation}{0}
 \setcounter{theorem}{0}
 \setcounter{lemma}{0}
 \setcounter{definition}{0}
 \setcounter{corollary}{0}
 \setcounter{proposition}{0}
\noindent {\large \bf 1. Introduction and Preliminaries}
\vspace{3mm}
\par
  A Galton-Watson process in which individuals have variable lifetimes is a discrete time stochastic process $\{Z_n;n\geq 0\}$ on $\mathbb{Z}_+=\{0,1,\cdots\}$. Let $\{p_k;k\geq 0\}$ and $\{h_k;k\geq 0\}$ denote the offspring distribution function and the lifetime distribution function, and let $f(s)$ and $g(s)$ denote the p.g.f of $\{p_k;k\geq 0\}$ and $\{h_k;k\geq 0\}$ respectively, $m$ the offspring mean and $l$ the lifetime mean, that is,
   \begin{equation}\label{ml}
   f(s)=\sum_{k=0}^{\infty}p_ks^{k},~~~~g(s)=\sum_{k=0}^{\infty}h_ks^{k},~~~m=\sum_{k=0}^{\infty}kp_k~~~\text{and}~~~l=\sum_{k=0}^{\infty}kh_k.
   \end{equation}
  Further, we define
   \begin{equation}\label{e1.1}
       q_k=
\begin{cases}
\sum_{i=k}^{\infty}h_i/\sum_{i=k-1}^{\infty}h_i, \ \ \ if~ \sum_{i=k-1}^{\infty}h_k \neq 0\\
0, \ \ \ \ \ \ \ \ \ \ \ \ \ \ \ \ \ \ \ \ \ \ \ \ \ \  if~ \sum_{i=k-1}^{\infty}h_k=0,
\end{cases}
\end{equation}
provide $\prod_{i=1}^{0}=1$. It is easy to see that
\begin{equation}\label{e1.2}
       h_k=(1-q_{k+1})\prod_{i=1}^{k}q_i.
\end{equation}
If $L$ denotes the life length of an individual, then $h_k=\mathbb{P}(L=k)$ and $q_k=\mathbb{P}(L\geq k|L\geq k-1)$.
\par
 The process can be regarded as representing an evolving population of individuals. It starts at time $0$ with $Z_0$ newborn individuals, at the end of first season, each individual produces $k$ ($k\geq 0$) offspring and remains alive with probability $q_1p_k$, or dies with probability $1-q_1$. At the end of $n$th season, an individual of age $j<n$ (alive for $j$ seasons) has a probability $q_jp_k$ of producing $k$ ($k\geq0$) offspring and remains alive, or it has a probability $1-q_j$ of immediate death. When $h_0>0$, it means there are individuals that die at end of their first season without any descendants.
 \par
 Whittle \cite{Whittle} was first to consider Galton-Watson processes in which individuals have variable lifetimes and gave the following extinction criterion.
 \begin{proposition}\label{pro} The extinction probability $q$ of a Galton-Watson process in which individuals have lifetimes $\{Z_n;n\geq0\}$ is the smallest nonnegative root of $g(f(s))=s$ in $[0,1]$, where $f(s)$ and $g(s)$ are defined in (\ref{ml}). Moreover, $q=1$ if and only if $f'(1)g'(1)\leq1$.
 \end{proposition}
  \par
  In this paper, we look at the process $\{Z_n;n\geq0\}$ as a Galton-Watson process with countably infinitely many types (we use the abbreviation GWCIMT to represent such process).
 Consider a GWCIMT $\{\emph{\textbf{Z}}_{n}=(Z_n^{(1)},Z_n^{(2)},Z_n^{(3)},\cdots); n\geq 0\}$ defined on $\mathbb{Z}_+^{\infty}$, where $Z_{n}^{(t)}$ represents the number of individuals of type $t$ particle at the $n$th generation, and $t$ in the type-set $\chi=\{1,2,3,\cdots\}$. The transition function is
\begin{eqnarray*}\
   \mathbb{P}(\emph{\textbf{i}},\emph{\textbf{j}})&=& \mathbb{P}(\emph{\textbf{Z}}_{n+1}=\emph{\textbf{j}}|\emph{\textbf{Z}}_{n}=
   \emph{\textbf{i})}\ \ \ \ \ \ \ \ \ \ \ \ \ \emph{\textbf{i}},\emph{\textbf{j}} \in \mathbb{Z}_+^{\infty} \\
   &=& \text{the coefficient of }\textbf{s}^{\emph{\textbf{j}}}\text{ in } [\emph{\textbf{f}}(\emph{\textbf{s}})]^{\emph{\textbf{i}}},
\end{eqnarray*}
where $\emph{\textbf{f}}(\emph{\textbf{s}})=(\emph{f}^{(1)}(\emph{\textbf{s}}),\emph{f}^{(2)}(\emph{\textbf{s}}),\cdots)$ and
\begin{eqnarray*}\
\emph{f}^{(i)}(\emph{\textbf{s}})=\sum_{\emph{\textbf{j}}\in \mathbb{Z}_+^{\infty}}p_{\emph{\textbf{j}}}^{i}\prod_{k\in\chi}s_k^{j_k}.~~~~~~~~\textbf{s}\in [0,1]^{\chi}
\end{eqnarray*}
 $p_{j_{1},j_{2},j_{3}\cdots}^{(i)}$ is the probability that a type $i$ parent individual produces $j_{1}$ particles of type $1$, $j_{2}$ individuals of type $2$,$\cdots$, $j_{n}$ individuals of type $n$ and so on.
 \par
Naturally, if we think of an individual of age $i$ as type $i+1$ individual, then the only possible transition is from $k$th type to $1$th type and $(k+1)$th type for $k>0$. Hence we can write out the specific expression of $p_{\emph{\textbf{j}}}^{i}$, that is
\begin{equation}\label{tran}
  p_{j_{1},j_{2},j_{3}\cdots}^{(i)}=
\begin{cases}
  1-q_{i}, \ \ \ \ \ \ \ \ \ \ \ \ \ \ \ j_{1},j_{2},\cdots=0 \\
  q_{i}p_{k},\ \ \ \ \ \ \ \ \ \ \ \ \ \ \ \ \ \ j_{1}=k,j_{i+1}=1,j_{2},\cdots,j_{i},j_{i+2},\cdots=0 \\
  0,\ \ \ \ \ \ \ \ \ \ \ \ \ \ \ \ \ \ \ \ \ otherwise.
\end{cases}
\end{equation}
We can also calculate the mean progeny matrix $\emph{\textbf{M}}=(m_{ij};i,j\geq 1)$ of $\{\emph{\textbf{Z}}_n;n\geq0\}$ with
\begin{equation}\label{mean}
m_{ij}=
\begin{cases}
mq_i,\ \ \ & \mbox{if}~~~j=1 \\
q_i,\ \ \  & \text{if}~~~j=i+1 \\
0, \ \ \ \ & otherwise,
\end{cases}
\end{equation}
where $m_{ij}=\mathbb{E}Z_{1j}^{(i)}$ represents the expected number of type $j$ offspring of a single type $i$ individual in one generation. We give a natural classification of Galton-Watson process in which individuals have variable lifetimes according to the extinction criterion of Proposition \ref{pro}.
\begin{definition}\
 A Galton-Watson process with countably infinitely many types is called a Galton-Watson process in which individuals have variable lifetimes if its transition function satisfies (\ref{tran}), where $\{q_k;k>0\}$ and $\{p_k;k\geq0\}$ are defined in (\ref{ml}) and (\ref{e1.1}). Furthermore, the process is called subcritical if $ml<1$, critical if $ml=1$, and supercritical if $ml>1$, where $m,l$ are defined in (\ref{ml}).

\end{definition}
\par

 Obviously, if $h_1=1$, then $\{Z^{(1)}_n; n\geq 0\}$ is actually a classical Galton-Watson process.
If the lifetimes variable of each individual is bounded (i.e., there exists $k>0$, $h_i=0$ for all $i\geq k$), then $\{\emph{\textbf{Z}}_{n}; n\geq0\}$ degenerates to a classical multi-type Galton-Watson process with a finite number of types. To avoid trivialities, we make the following basic assumptions throughout this paper:
\begin{flalign*}
&\emph{\bf \text{A1:}}~\emph{\textbf{Z}}_0=(1,0,0,\cdots).&\\
&\emph{\bf \text{A2:}}~q_i>0~~\text{for all}~~i>0.&\\
&\emph{\bf \text{A3:}}~ 0<m<\infty.&
\end{flalign*}
\par
    A matrix $\emph{\textbf{A}}$ is said to be irreducible if for any pair $(i,j)$ there exists $n\geq 1$ such that $(\emph{\textbf{A}}^{n})_{ij}>0$. A GWCIMT is said to be irreducible if and only if its mean progeny matrix is irreducible, then the mean matrix $\emph{\textbf{M}}$ of $\{\emph{\textbf{Z}}_n; n\geq 0\}$ is irreducible consequently by assumption A2 and (\ref{mean}).

\par
 When the number of types is finite and the mean matrix of the process is irreducible, it is well-known that the extinction probability vector $\emph{\textbf{q}}<\textbf{1}$ if and only if the Perron-Frobenius eigenvalue, or in another words the spectral radius of mean matix is strictly greater than $1$ (see \cite{Athreya1972Branching}). When the type is countably infinite, the analogue of Perron-Frobenius eigenvalue often refers to the convergence radius of the power series $\sum_{k\geq 0}r^k(\emph{\textbf{A}}^k)_{ij}$, where $\emph{\textbf{A}}$ is the mean progeny matrix of the process. For any irreducible matrix $\emph{\textbf{A}}$, Vere-Jones \cite{Jone} proves that there exists a positive number $\tau$,
 \begin{equation}\label{e4.1}
\lim\limits_{k\rightarrow\infty}\{(\emph{\textbf{A}}^k)_{ij}\}^{1/k}=1/\tau.
\end{equation}
 for any pair $(i,j)$.
  $\tau$ and $1/\tau$ is often referred to as the convergence radius and the convergence norm of $\emph{\textbf{A}}$ respectively. Note that the convergence norm is equivalent to spectral radius of $\emph{\textbf{A}}$ in the finite type case. Sagitov \cite[Definition 7]{Sagitov2013Linear} gives a natural classification for GWCIMT according to the asymptotic properties of its mean matrix $\emph{\textbf{A}}^{n}=(a_{ij}^{(n)})_{i,j=1}^{\infty}$ as $n\rightarrow\infty$.
\par
\begin{definition}\ \label{def}
An irreducible matrix $\emph{\textbf{A}}$ is called $\tau\emph{\text{-recurrent}}$ or $\tau\emph{\text{-transient}}$ depending on the divergence or the convergence of the series $\sum_{k\geq 0}\tau^k(\emph{\textbf{A}}^k)_{ij}$. A $\tau\emph{\text{-recurrent}}$ matrix $\emph{\textbf{A}}$ is called $\tau\emph{\text{-positive}}$ if $\lim\limits_{k\rightarrow\infty}\tau^k(\emph{\textbf{A}}^k)_{ij}>0$ for some pair $(i,j)$ and hence for all pairs $(i,j)$, and it is called $\tau\emph{\text{-null}}$ if this limit is zero. Furthermore, a Galton-Watson process with countably many types is said to be $\emph{\text{transient}}$ \{$\emph{\text{positive recurrent}}$, $\emph{\text{null recurrent}}$\} if its mean matrix is $\tau\emph{\text{-transient}}$ \{$\tau\emph{\text{-positive}}$, $\tau\emph{\text{-null}}$\}.
\end{definition}
\par
  For a general GWCMIT with offspring generating function $\emph{\textbf{g}}(\emph{\textbf{s}})$, there are two different types of extinction: global extinction $\emph{\textbf{q}}$ and partial extinction $\tilde{\emph{\textbf{q}}}$ whose entries contain the probability of global extinction and partial extinction respectively. $\emph{\textbf{q}}$ and $\tilde{\emph{\textbf{q}}}$ are both the solution of $\emph{\textbf{g}}(\emph{\textbf{s}})=(\emph{\textbf{s}})$ and it is possible for the population of each type to become extinct almost surely ($\tilde{\emph{\textbf{q}}}=\textbf{1}$) while the total population size explodes with positive probability ($\emph{\textbf{q}}<\textbf{1}$; see \cite{Hautphenne2013Extinction}). For the process $\{\emph{\textbf{Z}}_{n}; n\geq0\}$ considered in this paper, by (\ref{tran}), we can write $\emph{\textbf{f}}(\emph{\textbf{s}})=(\emph{\textbf{s}})$ componentwise as
 \begin{equation}\label{ite}
 1-q_i+q_is_{i+1}f(s_1)=s_i,~~~~~~(i\geq1).
 \end{equation}
 Iterating (\ref{ite}) to infinity and using (\ref{e1.2}) give $g(f(s_1))=s_1$. Then from Proposition \ref{pro} and (\ref{ite}), we know $\emph{\textbf{q}}\equiv\tilde{\emph{\textbf{q}}}$.

GWCIMTs have been investigated by several authors. Moyal \cite{J1962Multiplicative} considers a general type space and proves that the extinction probability is a solution of $\emph{\textbf{s}}=\emph{\textbf{g}}(\emph{\textbf{s}})$, where $\emph{\textbf{g}}(\cdot)$ is the progeny generating function of such processes. Hautphenne et al \cite{Hautphenne2013Extinction} give a sufficient condition for the appearance of the situation $\tilde{\emph{\textbf{q}}}=\textbf{1}$ and $\tilde{\emph{\textbf{q}}}<\textbf{1}$ in terms of truncated Galton-Watson processes and discuss the connection between the convergence norm of the mean progeny matrix and  the extinction criteria in the irreducible and reducible cases. Braunsteins and Hautphenne \cite{2017arXiv170602919B} consider a so-called lower Hessenberg branching processes with countably many types, which restrict individuals of type $i$ to give birth to type $j\leq i+1$ only. They prove the existence of a continuum of fixed points of the progeny generating function and show that the minimum of this continuum is $\emph{\textbf{q}}$ and the maximum is $\tilde{\emph{\textbf{q}}}$. Sagitov \cite{Sagitov2013Linear} considers a linear-fractional GWCIMT whose mean progeny matrix $\emph{\textbf{A}}$ has a form of
\[
\emph{\textbf{A}}=\emph{\textbf{H}}+a\emph{\textbf{H}}\emph{\textbf{1}}^{t}\emph{\textbf{g}},
\]
where $\emph{\textbf{H}}=(h_{ij})_{i,j=1}^{\infty}$ is a sub-stochastic matrix, $\emph{\textbf{g}}=(g_1,g_2,\cdots)$ is a proper probability distribution, $\emph{\textbf{1}}^t$ is the transpose of the row vector $\textbf{1}=(1,1,\cdots)$ and $a$ is a positive constant.
 For linear-fractional GWCIMT, Sagitov presents criteria for $\tau$-positive recurrence \cite[Theorem 8]{2017arXiv170602919B} which depend on the value of $d(s)=\sum_{n\geq1}d_ns^n$ in the point of its convergence radius, where
 \[
 d_n=\emph{\textbf{g}}\emph{\textbf{H}}^n\emph{\textbf{1}}^t.
 \]
 While it is not easy to know about the specific form of $d(s)$ to get the recurrence property of linear-fractional GWCIMT.
 \par
  In this paper, we first derive criteria for convergence radius $\gamma$ of $\emph{\textbf{M}}$ to be greater, smaller than and equal with 1. We give formula of calculating $\gamma$ and show that $\gamma$, in supercritical case, is actually the extinction probability of a Galton-Watson process, and the method of which are more efficient than Lemma \ref{L3.1} below. Next, we give clear criteria for $\emph{\textbf{M}}$ to be $\gamma$-transient, $\gamma$-positive and $\gamma$-null recurrent from which the ergodic property of the process is discussed. The criteria only rely on the properties of lifetime probability generating function $g(s)$ which are more clear and easier to be verified than the criterion of linear-fractional GWCIMT. Labeling each individual to translate the process to GWCIMT helps to know more about the evolution of Galton-Watson process in which individuals have variable lifetimes. We can see the exact geometric growth of total population size of $\{\emph{\textbf{Z}}_n;n\geq0\}$ in supercritical case due to $\gamma$. At last, an example of geometric law of lifetimes is given to illustrate our results.
\par
\vspace{5mm}
 \setcounter{section}{2}
 \setcounter{equation}{0}
 \setcounter{theorem}{0}
 \setcounter{lemma}{0}
 \setcounter{definition}{0}
 \setcounter{corollary}{0}
 \setcounter{proposition}{0}

\noindent {\large \bf 2. Results and Proofs}
 \vspace{3mm}
 \par
  \par
\par
 Denote convergence norm of $\emph{\textbf{M}}$ by $\rho$ then $\rho=\gamma^{-1}$, where $\emph{\textbf{M}}$ is defined in (\ref{mean}). Further, let vector $\emph{\textbf{v}}=\{v^{(i)}\}_{i>0}$ and $\emph{\textbf{u}}^{T}=\{u^{(i)}\}_{i>0}$ be the $\gamma$-invariant measure and vector of $\emph{\textbf{M}}$ respectively, with $\emph{\textbf{u}},\emph{\textbf{v}}>\textbf{0}$ and
\begin{eqnarray*}
\gamma\emph{\textbf{v}}\emph{\textbf{M}}=\emph{\textbf{v}},~~~~~ \gamma\emph{\textbf{M}}\emph{\textbf{u}}=\emph{\textbf{u}},
\end{eqnarray*}
where we write $\emph{\textbf{v}}>\textbf{0}$ to indicate that $v^{(i)}\geq 0$ for all $i>0$ but at least one component strictly greater.
Denote the $(k\times k)$ northwest corner truncation matrix of $\emph{\textbf{M}}$ by $\emph{\textbf{M}}^{(k)}$, the spectral radius of $\emph{\textbf{M}}^{(k)}$ by $\rho_k$.
\par
 It is known that $\tilde{\emph{\textbf{q}}}<1$ if $\rho>1$ and $\tilde{\emph{\textbf{q}}}=1$ if $\rho\leq1$ in the irreducible case (see \cite{Hautphenne2013Extinction}). Convergence norm $\rho$ plays a key role in investigating Galton-Watson processes with countably many types. However, it is not easy to evaluate $\rho$. In this paper we provide a more efficient method to evaluate the convergence norm of $\emph{\textbf{M}}$ (an unique solution of an equation which relies on the specific form of $\emph{\textbf{M}}$) than Lemma \ref{L3.1} below (approximating method), when $\emph{\textbf{M}}$ is the mean progeny matrix of a Galton-Watson process in which individuals have variable lifetimes. To do so, we introduce the following two Lemmas which correspond to Theorem 6.8 and 6.4 in Seneta \cite{Seneta2006Non}.
\begin{lemma}\label{L3.1}
If the matrix $\textbf{M}^{(k)}$ is irreducible for all $k>0$, then $\rho_k\uparrow\rho$. \hfill $\Box$
\end{lemma}
\begin{lemma}\label{L3.2} The matrix $\textbf{M}$ is $\gamma$-positive if and only if the $\gamma$-invariant measure and $\gamma$-invariant vector satisfy
\begin{eqnarray*}
\textbf{v}\textbf{u}=\sum_{i>0}u^{(i)}v^{(i)}<\infty.
\end{eqnarray*}\hfill $\Box$
\end{lemma}
Define $Q_k=\prod_{i=1}^{k}q_i$ if $k\geq 1$ and $Q_0=1$, $F(s)=m\sum_{j=1}^{\infty}Q_js^j$. Denote the radius of convergence of the power series $F(s)$ by $R$. Noting that $Q_k=P(L>k)$ by (\ref{qk}) below, then $F(1)=ml$. Now we state our main theorems.
\begin{theorem} \label{th3.3}
(1) If $ml<1$ and $F(R)\geq1$, then $\gamma=\rho^{-1}>1$ is the unique solution of $F(s)=1$ in $(1,R]$.\\
(2) If $ml<1$ and $F(R)<1$, then $\rho=R^{-1}$.\\
(3) If $ml=1$, then $\rho=1$.\\
(4) If $ml>1$, then $\rho>1$ and $\gamma=\rho^{-1}$ is the extinction probability of a Galton-Watson process whose offspring generating function is
\begin{displaymath}
B(s)=\frac{1+mg(s)}{1+m}.
\end{displaymath}

\end{theorem}
\begin{proof} The characteristic polynomial of $\emph{\textbf{M}}^{(k)}$ is
\begin{eqnarray*}
f^{(k)}(
\lambda)=|\lambda E-\emph{\textbf{M}}^{(k)}|&=&\left|\begin{array}{ccccc}
    \lambda-mq_1 &    -q_1    & 0       & \cdots  &0 \\
    -mq_2        &    \lambda & -q_2    & \cdots  &0 \\
    -mq_3        &    0       & \lambda & \cdots  &0 \\
    \vdots       &    \vdots  & \vdots  & \ddots  &0 \\
    -mq_{k-1}    &    0       & 0       & \cdots  &-q_{k-1}\\
    -mq_{k}      &    0       & 0       & \cdots  &\lambda \\
\end{array}\right| \\
                           &=&(-mq_k)(-1)^{k+1}(-q_1(-q_2)\cdots(-q_{k-1}))+\lambda(-1)^{2k}f^{(k-1)}(\lambda) \\
                           &=&\lambda f^{(k-1)}(\lambda)-mq_1q_2\cdots q_{k},
\end{eqnarray*}
\par
 where $E$ is the identity matrix. Let $s=\lambda^{-1}$, $\phi_k(s)=f^{(k)}(s^{-1})$. Hence the recursion becomes $s\phi_{k}(s)=\phi_{k-1}(s)-msQ_k$ for $k>1$, i.e., $s^j\phi_j(s)-s^{j-1}\phi_{j-1}(s)=-mQ_js^{j}$ valid for $j>1$. Now $\phi_1(s)=s^{-1}-mQ_1$, so summing gives
$$
s^k\phi_k(s)=1-m\sum_{j=1}^{k}Q_js^j,~~~~~~(k>1).
$$
Denote the subtracted term on right-hand side by $F_k(s)$.
\par
It follows that the positive zeros of $f^{(k)}(\lambda)$ are in one to one correspondence with the positive solutions of
\begin{equation}\label{equ}
F_k(s)=1.
\end{equation}
Clearly $F_k(0)=0$ and $F_k(s)$ is a polynomial which is convex increasing in $[0,\infty)$, hence (\ref{equ}) has exactly one positive root, denoted by $s_k$. It is easy to see that
\begin{equation}\label{qk}
Q_k=\sum_{j=k}^{\infty}h_j,
\end{equation}
which implies
$$
F_k(1)=m\sum_{j=1}^{k}Q_j=m\sum_{j\geq1}\sum_{i\geq j}h_i-r_k,~~~\mbox{where}~r_k:=m\sum_{j>k}\sum_{i\geq j}h_i.
$$
 So if $ml\leq 1$, then $F_k(1)\leq 1$ for all $k$ and hence (\ref{equ}) has a unique solution in $[1,\infty)$. The solution $s_k=1$ holds if and only if $ml=1$ and $r_k=0$ which contradicts assumption A2. Next, when $ml>1$, it is clear that $r_k\downarrow0$ as $k\uparrow\infty$. So there exists a positive number $N$, if $k>N$, $F_k(1)>1$ and then (\ref{equ}) has a unique solution $s_k\in(0,1)$. If $k<N$, then $F_k(1)<1$ and (\ref{equ}) has a unique solution $s_k\in(1,\infty)$.
\par
According to Lemma \ref{L3.1}, the convergence norm of $\emph{\textbf{M}}$ is $\rho=\lim_{k\rightarrow\infty}\rho_k=\lim_{k\rightarrow\infty}s_k^{-1}$. Hence we need to consider  $F(s)=\lim_{k\rightarrow\infty}F_k(s)$. Let $\hat{s}$ be the unique solution of
\begin{equation}\label{eq1}
F(s)=1,
\end{equation}
if there is a solution.
\par
Noting that $F(1-)=ml$, then there clearly is such a solution if $ml>1$ and $\hat{s}\in(0,1)$. Hence we first consider the case $ml>1$ and $s\in(0,1)$. It follows from the definition of $F_k(s)$ and (\ref{qk}) that
\[
F(s)=m\sum_{j\geq1}s^j\sum_{i\geq j}h_i=ms\frac{1-g(s)}{1-s},
\]
where $g(s)=\sum_{i\geq0}h_is^i$. Hence the limiting form of $(\ref{equ})$ becomes $1-s=ms(1-g(s))$, i.e.,
\begin{equation}\label{ps}
B(s):=\frac{1+msg(s)}{1+m}=s.
\end{equation}
\par
It can be verified that $B(s)$ is a probability generating function and hence (\ref{ps}) can be interpreted as the equation satisfied by the extinction probability of a Galton-Watson process whose offspring probability generating function is $B(s)$ and offspring mean is $m(1+l)/(1+m)$. This exceeds unity if and only if $ml>1$. We are assuming this condition at present, so $\hat{s}$ is this extinction probability (see \cite{Athreya1972Branching}). Also, $\hat{s}>B(0)=(1+m)^{-1}$, implying that $\rho\in(1,1+m)$. Then $\emph{(4)}$ of the theorem is proved.
\par
Now there is an important point to be made here because (\ref{ps}) also has the solution $s=1$ for any value of $ml$. This solution is an artefact of the multiplication by $1-s$ of (\ref{equ}) required to produce (\ref{ps}). This addition solution is irrelevant to the problem.
\par
If $ml=1$, then (\ref{ps}) has the unique solution $\hat{s}=1$. Hence $\rho=1$ in this case and $\emph{(3)}$ follows.
\par
In the case that $ml<1$ it is not clear whether (\ref{eq1}) has a solution. Recall that $F(1)=ml$ implies $R\geq1$. Hence if $R>1$ and $F(R-)\geq1$ (can be $\infty$), then it is clear that there is a unique solution $\hat{s}=\rho^{-1}\in(1,R]$ and $\emph{(1)}$ follows.
\par
However, if $R>1$ and $F(R-)<1$, or $R=1$, (\ref{eq1}) has no solution. In these cases, $\rho=\lim_{k\rightarrow\infty}\rho_k$ and $\rho_k^{-1}=s_k\downarrow R$ from the definition of $F_k(s)$, implying $\rho=1/R$. Then $\emph{(2)}$ is proved and the proof of the theorem is complete.
\hfill $\Box$
\end{proof}
 \begin{remark}\label{mar1}
In case (2) of the above theorem, (\ref{eq1}) has no solution and hence the $\gamma$-invariant vector and measure do not exist.
\end{remark}
\par
In the following theorem, we derive the ergodic property of $\{\emph{\textbf{Z}}_n;n\geq0\}$ due to Definition \ref{def}.
\begin{theorem}\label{th3.4}
(1) If $ml<1$ and $F(R)<1$, then $\{\textbf{Z}_n;n\geq0\}$ is transient. \\
(2) If $ml<1$ and $F(R)\geq1$, or $ml=1$ and $g''(1-)<\infty$, or $ml>1$, then $\{\textbf{Z}_n;n\geq0\}$ is positive recurrent.\\
(3) If $ml=1$ and $g''(1-)=\infty$, then $\{\textbf{Z}_n;n\geq0\}$ is null recurrent.
\end{theorem}
\begin{proof}\ From Seneta \cite[Theorem 6.2]{Seneta2006Non}, for a $\tau$-recurrent matrix $\emph{\textbf{A}}$, there always exists $\tau$-invariant measure. Then by Remark \ref{mar1} and Theorem \ref{th3.3}, $\emph{\textbf{M}}$ is $\gamma$-transient if $ml<1$, $R>1$ and $F(R-)\leq1$, or $ml<1$ and $R=1$. Hence $\{\emph{\textbf{Z}}_n; n\geq 0\}$ is transient in these cases. We assume the existence of $\hat{s}=\rho^{-1}$ where $\hat{s}$ solves (\ref{eq1}) in the rest of proof.
\par
Without loss of generality, let $u^{(1)}=1$, $v^{(1)}=1$. Then by $\emph{\textbf{M}}\emph{\textbf{u}}=\rho\emph{\textbf{u}}$ and $\emph{\textbf{v}}\emph{\textbf{M}}=\rho\emph{\textbf{v}}$, one can present that
\begin{eqnarray}\label{e3.10}
u^{(k)}=\cfrac{f^{(k-1)}(\rho)}{q_1q_2\cdots q_{k-1}},~~~~~~~~~~~~
v^{(k)}=\cfrac{q_1q_2\cdots q_{k-1}}{\rho^{k-1}}.
\end{eqnarray}
Noting that the derivation of the explicit form (\ref{e3.10}) shows that they are the unique outcomes starting from $u^{(1)}=v^{(1)}=1$, which implies that $\emph{\textbf{M}}$ is $\gamma$-recurrent by Kendall's Theorem (See \cite{Continuous-time Markov chains}).
\par
Next, consider the infinite sum
\begin{eqnarray}
\emph{\textbf{v}}\emph{\textbf{u}}&=&\sum_{i\geq1}v^{(i)}u^{(i)}=\sum_{i\geq1}\hat{s}^i\phi_i(\hat{s}) \nonumber \\
&=&\sum_{i\geq1}(1-m\sum_{j=1}^{i}Q_j\hat{s}^j)=m\sum_{i\geq1}\sum_{j>i}Q_j\hat{s}^j          \nonumber\\
&=&m\sum_{j=2}^{\infty}(j-1)Q_j\hat{s}^j=:S.\label{sum}
\end{eqnarray}
The first equality follows from (\ref{e3.10}) and the definition of $\hat{s}$ and $\phi_k(\cdot)$ and the second equality from (\ref{equ}). Then we only need to consider in which case $S$ is finite due to Lemma \ref{L3.2}.
\par
If $ml>1$, then $\hat{s}<1$ by Theorem \ref{th3.3} and hence the sum (\ref{sum}) is finite.
\par
From the definition of $Q_j$ and (\ref{sum}), we obtain
\begin{equation}\label{eee}
S=m\sum_{i=2}^{\infty}h_i\sum_{j=2}^{i}(j-2)\hat{s}^j.
\end{equation}
If $ml=1$, then $\hat{s}=1$ and clearly $S<\infty$ if and only if $g''(1-)<\infty$.\par
If $ml<1$ then $\hat{s}>1$ and reversing the order of summation in the inner sum of (\ref{eee}) it becomes
\[
S=m\sum_{i=2}^{\infty}ih_i\hat{s}^i\sum_{k=0}^{i-2}(1-(k+1)/i)\hat{s}^{-k}.
\]
The monotone convergence theorem implies that the inner sum converges to $\sum_{k\geq0}\hat{s}^{-k}<\infty$. It follows that $S<\infty$ if and only if $g'(\hat{s}-)<\infty$. Recall that $\hat{s}$ solves (\ref{eq1}), then
\begin{eqnarray}
1&=&m\sum_{i\geq1}Q_i\hat{s}^i=m\sum_{j\geq1}\hat{s}^j\sum_{i\geq j}h_i=m\sum_{i\geq1}h_i\sum_{j=1}^{i}\hat{s}^j \nonumber\\
&\geq&m\sum_{i\geq1}h_ii\hat{s}^{i-1}=mg'(\hat{s}).\label{ee2}
\end{eqnarray}
 The inequality follows from the fact that $\sum_{j=1}^is^j\geq is^{i-1}$ if $s\geq1$ for all $i>0$. Then the proof is complete according to the Lemma \ref{L3.2} and Definition \ref{def}.
\hfill $\Box$ \end{proof}
\par
In general, compared to the case that the mean matrix is $\tau$-transient or $\tau$-null in which the invariant vector and measure correspond to the convergence norm of mean matrix do not exist or can not be normalized, it is easier to discuss limit theorems of a GWCIMT when the mean matrix is $\tau$-positive. In this case, the mean matrix behaves like a finite matrix and there are several published results. Sagitov \cite[Propostion 9,10,11]{Sagitov2013Linear} presents basic asymptotic results for positive recurrent linear-fractional GWCIMT in supercritical, critical and subcritical case. Vatutin et al \cite[Theorem 3]{VV} extend mentioned Sagitov result in critical case by considering a general class of the reproduction generating function of individuals and without assuming the finiteness of the second moments for the offspring number of individuls. In supercritical and positive recurrent case, Moy \cite[Theorem 1]{1967Extensions} extends the theorem of Everret, Ulam and Harris to GWCIMT. It tells that if
\begin{equation}\label{last}
\sum_{i=1}^{\infty}v^{(i)}\mathbb{E}\big((\emph{\textbf{Z}}_1\emph{\textbf{u}})^2|\emph{\textbf{Z}}_0=\emph{\textbf{e}}_i\big)
\end{equation}
where $\emph{\textbf{e}}_i=(0,\cdots,0,1,0,\cdots)$ with 1 in the $i$th component, then there is a nonnegative real random variable $Y$ with a finite second moment such that for any $\emph{\textbf{w}}$ satisfying $\emph{\textbf{w}}\leq c\emph{\textbf{u}}^t$ for some positive constant $c$, the sequence $\{\rho^{-n}\emph{\textbf{Z}}_n\emph{\textbf{w}}^t\}$ converges in mean square to $Y\emph{\textbf{v}}\emph{\textbf{w}}^t$. Here $\emph{\textbf{v}}$ and \emph{\textbf{u}} are normalized $\tau$-invariant measure and vector with $\emph{\textbf{v}}\emph{\textbf{u}}=1$.
\par
From Theorem \ref{th3.3} and \ref{th3.4}, Galton-Watson process in which individuals have variable lifetimes covers all the three cases of GWCIMT. In the case of positive recurrence, one can easily apply results of Sagitov, Vatutin and Moy to get the corresponding properties for the process considered in this paper. The total population size of $\{\emph{\textbf{Z}}_n;n\geq0\}$ is actually the population size of $\{\emph{Z}_n;n\geq0\}$, hence constructing a Markovian GWCIMT helps understand the evolution of Galton-Watson process in which individuals have variable lifetimes. Here we only give the property in supercritical case. It is easy to know (\ref{last}) holds if and only if $f''(1)<\infty$, then we have the following proposition by applying Theorem $1$ of \cite{1967Extensions}.
\begin{proposition}
If $ml>1$ and $f''(1)<\infty$, then $\lim_{n\rightarrow\infty}\rho^{-n}\mathbb{E}(\textbf{1}\cdot\textbf{Z}_n)=(1+1/m)S^{-1}$, where $S$ is defined in (\ref{sum}).
\end{proposition}

\begin{example} Consider a geometric law of lifetimes. Let $g(s)=(1+l-ls)^{-1}$, where $l$ is a fixed positive number, then $g'(1)=l$. Denote the offspring mean by $m$ and assume $ml>1$. It follows from (\ref{ps}) that $\rho=(1+m)r$, where $r:=l/(1+l)$.
\par
calculation yields $Q_j=r^j$, so
\[
\sum_{j=1}^{k}Q_js^j=rs\frac{1-(rs)^k}{1-rs},
\]
and hence
\[
s^k\phi_k(s)=1-mrs\frac{1-(rs)^k}{1-rs}.
\]
It follows that
\[
\hat{s}^k\phi_k(\hat{s})=(1+m)^{-k}
\]
where $\hat{s}$ is the unique solution of (\ref{eq1}). We thus conclude that $\emph{\textbf{v}}\emph{\textbf{u}}$ converges, hence the process is positive recurrent in this case.
\par
Further calculation yields
\[
\lim_{n\rightarrow\infty}\rho^{-n}\mathbb{E}(\emph{\textbf{1}}\cdot\emph{\textbf{Z}}_n)=1.
\]
\end{example}
\section*{Acknowledgement}
\par
 This work is substantially supported by the National Natural Sciences Foundations of China (No. 11771452, No. 11971486) and Natural Sciences Foundations of Hunan (No. 2020JJ4674).

\end{document}